\documentclass[12pt]{amsart}
\usepackage{amsmath,amssymb}
\usepackage{graphicx,bm}
\usepackage{geometry}
\usepackage{hyperref}
\usepackage{amsthm}
\usepackage{verbatim} 
\usepackage{slashed}
\usepackage{cases}
\usepackage{color}
\numberwithin{equation}{section}

\newtheorem{thm}{Theorem}[section]

\newtheorem{cor}[thm]{Corollary}
\newtheorem{lem}[thm]{Lemma}

\newtheorem{prop}[thm]{Proposition}

\newtheorem{rmk}[thm]{Remark}

\newcommand{\pt}{\partial}
\DeclareMathOperator\arcsinh{arcsinh}

\begin{document}

\title[Existence of (Dirac-)harmonic Maps from Degenerating (Spin) Surfaces]{Existence of (Dirac-)harmonic Maps from  Degenerating (Spin) Surfaces}
\author{J\"urgen Jost, Jingyong Zhu}
%\thanks{$^1$The author would like to acknowledge the partial support by NSF DMS-1608782 for this research.}
%\thanks{$^2$The author wants to thank the support by China Scholarship Council for visiting University of California, Santa Cruz. }
\address{Max Planck Institute for Mathematics in the Sciences, Inselstrasse 22, 04103 Leipzig, Germany}
\email{jost@mis.mpg.de}
\address{Max Planck Institute for Mathematics in the Sciences, Inselstrasse 22, 04103 Leipzig, Germany}
\email{jizhu@mis.mpg.de}

%\date{}%
%\dedicatory{}%
%\commby{}%

% ----------------------------------------------------------------

\begin{abstract}
We study the existence of harmonic maps and Dirac-harmonic maps from degenerating surfaces to non-positive curved manifold via the scheme of Sacks and Uhlenbeck. By choosing a suitable sequence of $\alpha$-(Dirac-)harmonic maps from a sequence of suitable closed surfaces degenerating to a hyperbolic surface, we get the convergence and a cleaner energy identity under the uniformly bounded energy assumption. In this energy identity, there is no energy loss near the punctures. As an application, we obtain an existence result about (Dirac-)harmonic maps from degenerating (spin) surfaces. If the energies of the map parts also stay away from zero, which is a necessary condition, both the limiting harmonic map and Dirac-harmonic map are nontrivial.
\end{abstract}
\maketitle

% ----------------------------------------------------------------

\section{Introduction}  
The fundamental paper \cite{sacks1981existence} by Sacks and Uhlenbeck approached the theory of harmonic maps from a Riemann surface $M$ into a 
 Riemannian manifold $N$, that is 
 critical points $u:M\to N$  of the energy functional
\begin{equation}\label{1}
E(u)=\int_M|du|^2dvol_g,
\end{equation}
by perturbing the energy functional $E$ to the  functional 
\begin{equation}\label{2}
E_\alpha(u)=\int_M(1+|du|^2)^\alpha dvol_g,
\end{equation}
for $\alpha >1$ whose critical points are called $\alpha$-harmonic maps. In contrast to \eqref{1}, \eqref{2} satisfies a Palais-Smale condition so that standard methods apply. The difficult part then consists in controlling the limit $\alpha \to 1$. By studying this limit behavior of a sequence of $\alpha$-harmonic maps as $\alpha\searrow1$, they obtained the existence of harmonic maps and insight into the formation of bubbles.

Motivated by the supersymmetric nonlinear sigma model from quantum field theory, see \cite{jost2009geometry}, Dirac-harmonic maps from spin Riemann surfaces into Riemannian manifolds were introduced in \cite{chen2006dirac}. They are generalizations of the classical harmonic maps and harmonic spinors. From the variational point of view, they are critical points of a conformal invariant  action functional whose Euler-Lagrange equation is a coupled elliptic system consisting of a second order equation and a Dirac equation. Being generalizations of harmonic maps, it looks natural to extend the scheme of \eqref{2} to them. However, new difficulties arise. 

In fact, it turns out that the existence of  Dirac-harmonic maps from closed surfaces is a very difficult problem because the kernel of the Dirac operator is a linear space. Moreover, different from the Dirichlet problem, even if there is no bubble, the non-triviality of the limit is also an issue. Here, a solution is considered  trivial if the spinor part $\psi$ vanishes identically. So far, there are only a few results in this direction. See \cite{ammann2013dirac} and \cite{chen2015dirac} for uncoupled Dirac-harmonic maps (here uncoupled means that the map part is harmonic) based on index theory and the Riemann-Roch theorem, respectively. The problem has also been approached by linking theory, see \cite{jost2019alpha}. Recently, we are succeed in get an existence result by the heat flow method \cite{jost2019short}. Precisely, we use $\alpha$-Dirac-harmonic map flow starting from an initial map with nontrivial $\alpha$-genus to get nontrivial $\alpha$-Dirac-harmonic maps. Then a nontrivial Dirac-harmonic map comes from a sequence of nontrivial $\alpha$-Dirac-harmonic maps by blow-up analysis. This procedure can be viewed as an extension of Sacks-Uhlenbeck scheme to Dirac-harmonic maps. 

Given these existence results about (Dirac-)harmonic maps on closed surfaces, it is natural to consider the compactness. When the domain surfaces are fixed, the compactness problem is well-understood, see \cite{sacks1981existence} for harmonic maps and \cite{chen2005regularity}\cite{zhao2007energy} for Dirac-harmonic maps. When the domain surfaces vary, it is necessary to consider the degeneration of conformal structures on a Riemann surface. Topologically, the limit surface is obtained by collapsing finitely many simple closed geodesics. There are two types of collapsing curves. One is homotopically trivial, which corresponds to the bubbling near isolated singularities. The other one is homotopically nontrivial, which corresponds to the degeneration of complex structure. By the bubbling procedure in \cite{zhu2010harmonic}\cite{zhu2009dirac}, two generalized energy identities were proved. These identities tell us that energy may get lost from the formation of some necks that fail to converge to points.
Therefore, given a sequence of harmonic maps with bounded energy, we cannot assure the non-triviality of the limit harmonic maps from degenerating surfaces even if the energy uniformly stays away from zero. For the limit Dirac-harmonic map, the situation is worse. We cannot give a geometric condition to assure the non-triviality of the spinor because, so far, there is no  condition to get rid of nontrivial Dirac-harmonic spheres. 

In this paper, motivated by the existence of nontrivial $\alpha$-(Dirac-)harmonic maps from closed (spin) surfaces, we consider the (Dirac-)harmonic maps on degenerating surfaces. We first study the compactness of a sequence of $\alpha$-Dirac-harmonic maps from closed hyperbolic surfaces degenerating to a hyperbolic Riemann surface and get a nice energy identity. In this new energy identity, there is no energy loss on the degenerating cylinder, which turns out to be very important in proving the non-triviality of limit (Dirac-)harmonic maps.

\begin{thm}[Compactness and energy identity]
Let $(M_n,h_n,c_n,\mathfrak{S}_n)$ be a sequence of closed hyperbolic surfaces of genus $g>1$ degenerating to a hyperbolic Riemann surface $(M,h,c,\mathfrak{S})$ by collapsing finitely many pairwise disjoint simple closed geodesics $\{\gamma_n^j, j\in J\}$. Denote by $l_n^j$ the length of $\gamma_n^j$ and $\l_n:=\max\limits_{j\in J}\{l_n^j\}$. We choose a sequence of constants, $\{\alpha_n\in(1,2)\}$, such that 
\begin{equation}
\lim_{n\to\infty}\bigg(\frac{2l_n}{\sqrt{\sinh{\frac{l_n}{2}}}}\bigg)^{\alpha_n-1}=0.
\end{equation}
For each $n$, suppose that $(u_n,\psi_n)$ is an $\alpha_n$-Dirac-harmonic map from $(M_n,h_n,c_n,\mathfrak{S}_n)$ into a fixed compact manifold $N$ with nonpositive curvature and that it satisfies
\begin{equation}
E_{\alpha_n}(u_n)+E(\psi_n)\leq\Lambda,
\end{equation}
 for some positive constant $\Lambda$, where $E(\psi_n):=E(\psi_n,h_n,M_n)=\int_{M_n}|\psi|^4dvol_{h_n}$. Moreover, we assume that the first positive eigenvalue $ \lambda_1(h_n)$ of the usual Dirac operator $\slashed{\partial}_{h_n}$ stays away from zero, i.e.
 \begin{equation}\label{1st eigenvalue}
 \lambda_1(h_n)\geq c_0
 \end{equation}
 for some positive constant $c_0>0$. Then there exists a Dirac-harmonic map $(u,\psi):(M,h,c,\mathfrak{S})\to N$ such that, after a selection of a subsequence,
 \begin{equation}\label{loc convergence}
 (u_n,\psi_n)\to(u,\psi) \ \text{in} \ C^{\infty}_{loc}(M)\times C^{\infty}_{loc}(M)
 \end{equation}
 and 
 \begin{equation}
 \lim_{n\to\infty}E(u_n,h_n,M_n)=E(u,h,M),
 \end{equation}
  \begin{equation}
 \lim_{n\to\infty}E(\psi_n,h_n,M_n)=E(\psi,h,M).
 \end{equation}
 \end{thm}
 
In the special case of  harmonic maps, the result becomes 
 \begin{thm}
Let $(\Sigma_n,h_n,c_n)$ be a sequence of closed hyperbolic surfaces of genus $g>1$ degenerating to a hyperbolic Riemann surface $(\Sigma,h,c)$ by collapsing finitely many pairwise disjoint simple closed geodesics $\{\gamma_n^j, j\in J\}$. Denote by $l_n^j$ the length of $\gamma_n^j$ and $\l_n:=\max\limits_{j\in J}\{l_n^j\}$. We choose a sequence of constants, $\{\alpha_n\in(1,2)\}$, such that 
\begin{equation}
\lim_{n\to\infty}\bigg(\frac{2l_n}{\sqrt{\sinh{\frac{l_n}{2}}}}\bigg)^{\alpha_n-1}=0..
\end{equation}
For each $n$, suppose that $u_n$ is an $\alpha_n$-harmonic map from $(\Sigma_n,h_n,c_n)$ into a fixed manifold $N$ with nonpositive curvature which satisfies
\begin{equation}
E_{\alpha_n}(u_n)\leq\Lambda,
\end{equation}
 for some positive constant $\Lambda$. Then there exists a harmonic map $u:(\Sigma,h,c)\to N$ such that, after a selection of a subsequence,
 \begin{equation}
 u_n\to u\ \text{in} \ C^{\infty}_{loc}(\Sigma)
 \end{equation}
 and 
 \begin{equation}
 \lim_{n\to\infty}E(u_n,h_n,\Sigma_n)=E(u,h,\Sigma).
 \end{equation}
\end{thm} 
Actually, such compactness and energy identities are also  true for $\alpha(>1)$-(Dirac-)harmonic maps, see Theorems \ref{id dh alpha} and \ref{id h alpha}.

Moreover, combining these results with the existence of $\alpha$-(Dirac-)harmonic maps in \cite{sacks1981existence,ammann2013dirac}, we get an existence result about (Dirac-)harmonic maps from degenerating (spin) surfaces. Moreover, if the energy of the map parts also stays away from zero, both the limiting harmonic map and Dirac-harmonic map are nontrivial.

 \begin{thm}[Existence of Dirac-harmonic maps from degenerating surfaces]\label{existence dh degenerate}
Let $(M_n,h_n,c_n,\mathfrak{S}_n)$ be a sequence of closed hyperbolic surfaces of genus $g>1$ degenerating to a hyperbolic Riemann surface $(M,h,c,\mathfrak{S})$ by collapsing finitely many pairwise disjoint simple closed geodesics $\{\gamma_n^j, j\in J\}$. Denote by $l_n^j$ the length of $\gamma_n^j$ and $\l_n:=\max\limits_{j\in J}\{l_n^j\}$. For each $n$, suppose that $u_n$ is a map from $M_n$ into a fixed manifold $N$ with nonpositive curvature which satisfies
\begin{equation}\label{energy bound}
E_{\alpha_n}(u_n)\leq\Lambda,
\end{equation}
\begin{equation}\label{ind}
{\rm ind}_{u_n^*TN}(M_n)\neq0,
\end{equation}
where $\Lambda$ is a positive constant, and $\{\alpha_n\in(1,2)\}$ satisfies
\begin{equation}\label{alpha_n}
\lim_{n\to\infty}\bigg(\frac{2l_n}{\sqrt{\sinh{\frac{l_n}{2}}}}\bigg)^{\alpha_n-1}=0.
\end{equation}
Moreover, we assume that the first positive eigenvalue $ \lambda_1(h_n)$ of the usual Dirac operator $\slashed{\partial}_{h_n}$ stays away from zero, i.e.
 \begin{equation}\label{special degeneration}
 \lambda_1(h_n)\geq c_0
 \end{equation}
 for some positive constant $c_0>0$. Then there exists a Dirac-harmonic map $(u,\psi):(M,h,c,\mathfrak{S})\to N$ such that
  \begin{equation}\label{nonvanish}
\psi\neq0.
 \end{equation}
 In addition, let the minimizing harmonic map  $u_n^{\min}$  in $[u_n]$  satisfy
 \begin{equation}
 \lim_{n\to\infty}E(u_n^{\min})>0.
 \end{equation}
 Then $u$ is not a constant.
 \end{thm}
 
  \begin{thm}
Let $(\Sigma_n,h_n,c_n)$ be a sequence of closed hyperbolic surfaces of genus $g>1$ degenerating to a hyperbolic Riemann surfaces $(\Sigma,h,c)$ by collapsing finitely many pairwise disjoint simple closed geodesics $\{\gamma_n^j, j\in J\}$. Denote by $l_n^j$ the length of $\gamma_n^j$ and $\l_n:=\max\limits_{j\in J}\{l_n^j\}$. For each $n$, suppose that $u_n$ is a map from $\Sigma_n$ into a fixed manifold $N$ with nonpositive curvature and satisfies
\begin{equation}
E_{\alpha_n}(u_n)\leq\Lambda,
\end{equation}
where $\Lambda$ is a positive constant, and $\{\alpha_n\in(1,2)\}$ satisfies
\begin{equation}
\lim_{n\to\infty}\bigg(\frac{2l_n}{\sqrt{\sinh{\frac{l_n}{2}}}}\bigg)^{\alpha_n-1}=0.
\end{equation}
 Then there exists a harmonic map $u_0:(\Sigma,h,c)\to N$, which can be extended to a smooth harmonic map on $(\bar{\Sigma},\bar{c})$.
 
 In addition, let the minimizing harmonic map  $u_n^{\min}$  in $[u_n]$  satisfy
 \begin{equation}
 \lim_{n\to\infty}E(u_n^{\min})>0.
 \end{equation}
 Then $u_0$ is not a constant.
 \end{thm}
 
Here is a remark on Theorem \ref{existence dh degenerate} above. First, unlike \cite{zhu2009dirac}, we do not make an assumption on  the type of the punctures of $M$. Therefore, our result gives an existence result for  Dirac-harmonic map on those Riemann surfaces with Ramond type punctures. Even in the Neveu-Schwarz case, in which the Dirac-harmonic map $(u,\psi)$ can be extended to a smooth Dirac-harmonic map $(\bar{u},\bar{\psi})$ on $(\bar{M},\bar{c},\bar{\mathfrak{S}})$,  the existence of a Dirac-harmonic map on $(\bar{M},\bar{c},\bar{\mathfrak{S}})$ does not directly follow from  the result in \cite{ammann2013dirac} because we do not know whether ${\rm ind}_{\bar{u}^*TN}(\bar{M},\bar{\mathfrak{S}})\neq0$. Second, the uniformly bounded energy assumption \eqref{energy bound} is possible. One can take a simple example where $u_n$ is constant in a fixed neighborhood of the degenerating curves. Third, the index assumption \eqref{ind} is used to assure the non-triviality of the spinor, i.e. \eqref{nonvanish}. In \cite{ammann2013dirac}, the authors discussed some situations to realize the assumption \eqref{ind}. See for example Corollary 10.3 and Theorem 10.6 in \cite{ammann2013dirac}. Last, if the limit surface $M$ has a discrete Dirac spectrum, then Pf\"affle \cite{pfaffle2005eigenvalues} proved that the usual Dirac operators $\slashed{\partial}_{h_n}$ are $(\Upsilon,\epsilon)$-spectral close to $\slashed{\partial}_h$. In this case, if we also assume that the dimension of the kernel of $\slashed{\partial}_{h_n}$ converges to the that of $\slashed{\partial}_h$, then 
the assumption \eqref{special degeneration} follows. This leads to the question  when the limit surface $M$ has a discrete spectrum. An answer was provided by B\"ar in \cite{bar2000dirac}. 

\begin{thm}\cite{bar2000dirac}
Let $M$ be a hyperbolic manifold of finite volume equipped with a spin structure. If the spin structure is nontrivial along all cusps, i.e.,  all the punctures of $M$ are of Neveu-Schwarz type, then the spectrum is discrete. In particular, if $M$ is a $2$- or $3$-manifold and has only one cusp, then the spectrum is always discrete.

Here the cusps of $M$ are approximated by degenerating tubes around short closed geodesics in $M_n$ of length $l_n^j$, $j\in J$, where $|J|$ is the number of cusps.
\end{thm}

The finite volume assumption follows from the assumption \eqref{energy bound} (see the proof of Theorem \ref{id dh}). Combining all these facts, we have the following corollary.

\begin{cor}
Let $(M_n,h_n,c_n,\mathfrak{S}_n)$ be a sequence of closed hyperbolic surfaces of genus $g>1$ degenerating to a hyperbolic Riemann surface $(M,h,c,\mathfrak{S})$ with only Neveu-Schwarz type punctures by collapsing finitely many pairwise disjoint simple closed geodesics $\{\gamma_n^j, j\in J\}$. Denote by $l_n^j$ the length of $\gamma_n^j$ and $\l_n:=\max\limits_{j\in J}\{l_n^j\}$. We assume

(i) the dimension of the kernel  of $\slashed{\partial}_{h_n}$ converges to  that of $\slashed{\partial}_h$,

(ii) $\{u_n\}$ is a sequence of maps from $M_n$ into a closed even-dimensional orientable nonpositive curved manifold $N$ with nontrivial pull-back bundle $u_n^*TN\to M_n$ and uniformly bounded $\alpha_n$-energy for $\alpha_n$ satisfying \eqref{alpha_n}.

 Then there exists a smooth Dirac-harmonic map $(u,\psi):(M,h,c,\mathfrak{S})\to N$, which can be extended to a smooth Dirac-harmonic map $(\bar{u},\bar{\psi})$ on $(\bar{M},\bar{c},\bar{\mathfrak{S}})$, such that
  \begin{equation}
\bar\psi,\psi\neq0.
 \end{equation}
 
 In addition, let   the minimizing harmonic map $u_n^{\min}$ in $[u_n]$  satisfy
 \begin{equation}
 \lim_{n\to\infty}E(u_n^{\min})>0.
 \end{equation}
 Then both $\bar{u}$ and $u$ are nontrivial.
\end{cor}

\begin{rmk}
When the target manifold $N$ is odd-dimensional, one can choose spin structures $\mathfrak{S}_n$ on $M_n$ and $\mathfrak{S}$ on $M$ such that the conclusion in the corollary above still holds.
\end{rmk}

The rest of paper is organized as follows: In Section 2, we collect some well-known facts about the hyperbolic Riemann surface theory and some lemmas about $\alpha$-Dirac-harmonic maps. In Section 3, we show the generalized energy identity for a sequence of $\alpha$-Dirac-harmonic maps from non-degenerating spin surfaces. In Section 4, we will prove all the compactness and existence results stated in this introduction.

\section{Preliminaries}    
In order to study the (Dirac-)harmonic maps from degenerating (spin) surfaces, we recall some basic knowledge about degenerating surfaces and refer to \cite{zhu2010harmonic,zhu2009dirac} for more details.
\subsection{Degenerating (spin) surfaces}
A Riemann surface is an orientable surface with a complex structure. A hyperbolic surface is an orientable surface with a complete Riemannian metric of constant curvature $-1$ having finite area. The topological type of a surface is determined by its signature $(g, k)$, where $k$ is the number of punctures and $g$ is the genus of the surface obtained by adding a point at each puncture. The type $(g, k)$ is called general if 
$2g+k>2$. Two surfaces of type $(g, k)$ are called equivalent if there exists a conformal diffeomorphism between them preserving the punctures (if there are any). The space of equivalence classes is called the moduli space $\mathcal{M}_{g,k}$ of Riemann surfaces of type $(g, k)$. 

Now we consider closed Riemann surfaces of genus $g>1$. Any such surface is of general type and it acquires a complete hyperbolic metric. Therefore, they all stay in $\mathcal{M}_g$. This moduli space is non-compact because the conformal structure can degenerate. The only process by which the conformal structure can degenerate is the shrinking of simple closed geodesics. Hence, when we vary these in $\mathcal{M}_g$, they can degenerate into a surface with punctures. Generally, a sequence of surfaces of the same type can degenerate into another surface of different type. This inspires us to study the existence of (Dirac-)harmonic maps from degenerating surfaces based on the known existence results on closed surfaces.

The natural way to compactify $\mathcal{M}_g$ is to allow the lengths of the geodesics to become zero and thus admit surfaces with nodes as singularities. Topologically, one cuts a given  surface along a collection of finitely many homotopically independent pairwise disjoint simple closed curves and pinches the cut curves to points. More precisely, let $\Sigma_0$ be a topological model surface and $\Gamma^J=\{\gamma^j,j\in J\}$ a possibly empty collection of finitely many pairwise disjoint, homotopically nontrivial, simple closed curves on $\Sigma_0$. Let $\tilde\Sigma$ be the surface obtained from $\Sigma_0$ by pinching all curves $\gamma^j$ to points $\mathcal{P}^j$. Next, we remove all $\mathcal{P}^j$ from $\tilde\Sigma$ and place a complete hyperbolic metric $h$ on the resulting surface $\Sigma=\tilde\Sigma\setminus\cup_{j\in J}\mathcal{P}^j$. For $j\in J$, we denote by $(\mathcal{P}^{j,1},\mathcal{P}^{j,2})$ a pair of punctures on $(\Sigma,h)$ corresponding to $\mathcal{P}^{j}$. Denote by $\bar\Sigma$ the surface obtained by adding a point at each puncture of $\Sigma$. Then the complex structure $c$ on $\Sigma$ that is compatible with the hyperbolic metric $h$ extends to a complex structure $\bar{c}$ on $\bar\Sigma$. We call $(\tilde\Sigma,h,\bar{c})$ a nodal surface. $(\bar\Sigma,\bar{c})$ is called the normalization of $(\tilde\Sigma,h,\bar{c})$ or $(\Sigma,h,c)$. $\bar\Sigma$ is of lower topological type.

Let $(\Sigma_n,h_n,c_n)$ be a sequence of closed hyperbolic Riemann surfaces of genus $g>1$. We say that $(\Sigma_n,h_n,c_n)$ converges to a nodal surface $(\tilde\Sigma,h,\bar{c})$ or $(\Sigma,h,c)$ if there exist possibly empty collections $\Gamma^J_n=\{\gamma^j_n,j\in J\}$ of finitely many pairwise disjoint simple closed geodesics on each $(\Sigma_n,h_n,c_n)$ and continuous maps $\tau_n:\Sigma_n\to\tilde\Sigma$ with $\tau_n(\gamma_n^j)=\mathcal{P}^j$ for $j\in J$ and each $n$, such that\\
(1) For any $j\in J$, the length $l(\gamma_n^j)=l_n^j\to0$ as $n\to\infty$,\\
(2) $\tau_n:\Sigma_n\setminus\cup_{j\in J}\gamma_n^j\to\Sigma$ is a diffeomorphism for each $n$,\\
(3) $\bar{h}_n:=(\tau_n)_*h_n\to h$ in $C^\infty_{loc}$ on $\Sigma$,\\
(4) $\bar{c}_n:=(\tau_n)_*c_n\to c$ in $C^\infty_{loc}$ on $\Sigma$.

 If $|J|>0$, we say that $(\Sigma_n,h_n,c_n)$ \textit{degenerates to a nodal surface} $(\tilde\Sigma,h,\bar{c})$ or $(\Sigma,h,c)$. Thus, the analysis of the degeneration of hyperbolic surfaces is reduced to the local behavior of pinched geodesics. The following collar lemma is a fundamental tool to analyze this localization.
 \begin{lem}\cite{zhu2010harmonic}
 Let $\gamma$ be a simple closed geodesic of length $l(\gamma)=l$ in a closed Riemann surface $\Sigma$ of genus $g>1$. Then there is a collar of area $\frac{2l}{\sinh(l/2)}$ around $\gamma$ which is isometric to 
 \begin{equation}
 Z=\bigg\{re^{iw}\in\mathbb{H}: 1\leq r\leq e^l, \ \arctan\bigg(\sinh\bigg(\frac{l}{2}\bigg)\bigg)<w<\pi-\arctan\bigg(\sinh\bigg(\frac{l}{2}\bigg)\bigg)\bigg\},
 \end{equation}
 where $\gamma$ corresponds to $\{re^{i\pi/2}\in\mathbb{H}: 1\leq r\leq e^l\}$, and the lines $\{r=1\}$, $\{r=e^l\}$ are identified via $z\to e^lz$. 
 \end{lem}
 Topologically, this collar neighborhood is a cylinder. Under the conformal transformation
 \begin{equation}
 re^{iw}\to(t,\theta)=\bigg(\frac{2\pi}{l}w,\frac{2\pi}{l}\log r\bigg),
 \end{equation}
 the collar $Z$ is isometric to  following cylinder
 \begin{equation}
 \begin{split}
 P=\bigg\{(t,\theta): &\frac{2\pi}{l}\arctan\bigg(\sinh\bigg(\frac{l}{2}\bigg)\bigg)<t<\frac{2\pi}{l}\bigg(\pi-\arctan\bigg(\sinh\bigg(\frac{l}{2}\bigg)\bigg)\bigg),\\
 &0\leq\theta\leq2\pi\bigg\}
 \end{split}
 \end{equation}
 with metric
 \begin{equation}
  ds^2=\bigg(\frac{l}{2\pi\sin{\frac{lt}{2\pi}}}\bigg)^2(dt^2+d\theta^2),
  \end{equation}
  where $\gamma\subset Z$ corresponds to $\{t=\pi^2/l\}\subset P$, and the lines $\{\theta=0\}$, $\{\theta=2\pi\}$ are identified. In these coordinates, we have
  \begin{equation}\label{inj}
  \sinh({\rm inj}(t,\theta))\sin\bigg(\frac{lt}{2\pi}\bigg)=\sinh\bigg(\frac{l}{2}\bigg),
  \end{equation}
  where ${\rm inj}(t,\theta)$ is the injectivity radius at the point $(t,\theta)\in P$. Sometimes, we also denote by ${\rm inj}(x;h)$ the injectivity radius at the point at $x$ with respect to the metric $h$.
  
To  better understand the degeneration near the punctures, we recall the thick-thin decomposition (see \cite{hummel1997gromov}). Let $(\Sigma,h)$ be a hyperbolic surface of type $(g,k)$. For $0<\delta<\arcsinh 1$, define the $\delta$-thin part of $(\Sigma,h)$ as the set of points at which the injectivity radius is less than $\delta$, and the $\delta$-thick part as its complement. The following results show what the components of $\delta$-thin part of $(\Sigma,h)$ look like.
\begin{lem}
Let $\Sigma$ be a hyperbolic surface of type $(g,k)$. Then the simple closed geodesics in $\Sigma$ of length smaller then $2\arcsinh 1$ are pairwise disjoint and there are at most $3g-3+k$ of them.
\end{lem}

\begin{prop}
Let $(\Sigma,h)$ be a hyperbolic surface of type $(g,k)$ and $U\subset\Sigma$ a component of $\{z\in\Sigma|{\rm inj}(z;h)<\arcsinh 1\}$. Then either

(1) $U$ contains a simple closed geodesic $\gamma$ of length $l=l(\gamma)<2\arcsinh 1$ and is isometric to 
\begin{equation*}
\bigg\{re^{iw}\in\mathbb{H}: 1\leq r\leq e^l, \ \arcsin\bigg(\sinh\bigg(\frac{l}{2}\bigg)\bigg)<w<\pi-\arcsin\bigg(\sinh\bigg(\frac{l}{2}\bigg)\bigg)\bigg\},
\end{equation*}
 where $\gamma$ corresponds to $\{re^{i\pi/2}\in\mathbb{H}: 1\leq r\leq e^l\}$, and the lines $\{r=1\}$, $\{r=e^l\}$ are identified via $z\to e^lz$.
 
 or
 
 (2) the closure of $U$ in $\Sigma$ is a standard puncture and hence isometric to 
 \begin{equation*}
\bigg\{z=x+iy\in\mathbb{H}: 0\leq x\leq1, \ y\geq\frac12\bigg\},
\end{equation*}
where the lines $\{x=0\}$, $\{x=1\}$ are identified via $z\to z+1$.
\end{prop}
  
  For the degeneration of spin surfaces, we also need to consider the spin structure. Let $(M_n,h_n,c_n,\mathfrak{S}_n)$ be a sequence of closed hyperbolic Riemann surfaces of genus $g>1$ with spin structures $\mathfrak{S}_n$. We assume that $(M_n,h_n,c_n)$ degenerates to a hyperbolic Riemann surface $(M,h,c)$ by collapsing $|J|$ ($1\leq|J|\leq3g-3$) pairwise disjoint simple closed geodesics on $M_n$. Let $(\bar{M},\bar{c})$ be the normalization of $(M,h,c)$. For each $n$, the diffeomorphism $\tau_n$ and the spin structure $\mathfrak{S}_n$ together determine a pull-forward spin structure on $M$. Since there are finitely many spin structures on a surface with punctures (c.f.\cite{zhu2009dirac}), by taking a subsequence, we can assume that the pull-forward of $\mathfrak{S}_n$
 is a fixed spin structure on $M$. Let us denote it by $\mathfrak{S}$. For each $j\in J$, the unit tangent vector field of $\gamma_n^j$ together with the corresponding unit normal vector field forms a section of the oriented orthonormal frame bundle $P_{SO(2)}$ of $M_n$. The spin structure $\mathfrak{S}_n$ is called trivial along $\gamma_n^j$ if this section lifts to a closed curve in $P_{{\rm spin}(2)}$; otherwise, it is nontrivial along $\gamma_n^j$. Therefore, $\mathfrak{S}$ is nontrivial or trivial along the pair of punctures $(\mathcal{P}^{j,1},\mathcal{P}^{j,2})$ if and only if $\mathfrak{S}_n$ is nontrivial or trivial along the geodesics $\gamma_n^j$ for all $n$. If the spin structure $\mathfrak{S}$ is nontrivial along all punctures on $M$, we say that all the punctures on $(M,\mathfrak{S})$ are of \textit{Neveu-Schwarz} type. In this case, the spin structure $\mathfrak{S}$ on $M$ extends to some spin structure $\bar{\mathfrak{S}}$ on $\bar{M}$. Furthermore, by the removable singularity theorem (c.f.\cite{zhu2009dirac}), any smooth Dirac-harmonic map from $(M,\mathfrak{S})$ to $N$ with finite energy extends to a smooth Dirac-harmonic map from $(\bar{M},\bar{\mathfrak{S}})$ to $N$.

In the rest of this section, we recall the definition of  Dirac-harmonic maps and collect some lemmas about $\alpha$-Dirac-harmonic maps which will be used later.
\subsection{Dirac-harmonic maps}
Let $(M, g)$ be a compact surface with a fixed spin structure. On the spinor bundle $\Sigma M$, we denote the Hermitian inner product by $\langle\cdot, \cdot\rangle_{\Sigma M}$. For any $X\in\Gamma(TM)$ and $\xi\in\Gamma(\Sigma M)$, the Clifford multiplication is skew-adjoint in the sense that 
\begin{equation}
\langle X\cdot\xi, \eta\rangle_{\Sigma M}=-\langle\xi, X\cdot\eta\rangle_{\Sigma M}.
\end{equation}
Let $\nabla$ be the Levi-Civita connection on $(M,g)$. There is a connection (also denoted by $\nabla$) on $\Sigma M$ compatible with $\langle\cdot, \cdot\rangle_{\Sigma M}$.  Choosing a local orthonormal basis $\{e_{\beta}\}_{\beta=1,2}$ on $M$, the usual Dirac operator is defined as $\slashed\partial:=e_\beta\cdot\nabla_\beta$, where $\beta=1,2$. Here and in the sequel, we use the Einstein summation convention. One can find more about spin geometry in \cite{lawson1989spin}.

Let $u$ be a smooth map from $M$ to another compact Riemannian manifold $(N, h)$ of dimension $n\geq2$. Let $u^*TN$ be the pull-back bundle of $TN$ by $u$ and consider the twisted bundle $\Sigma M\otimes u^*TN$. On this bundle there is a metric $\langle\cdot,\cdot\rangle_{\Sigma M\otimes u^*TN}$ induced from the metric on $\Sigma M$ and $u^*TN$. Also, we have a connection $\tilde\nabla$ on this twisted bundle naturally induced from those on $\Sigma M$ and $u^*TN$. In local coordinates $\{y^i\}_{i=1,\dots,n}$, the section $\psi$ of $\Sigma M\otimes u^*TN$ is written as 
$$\psi=\psi_i\otimes\partial_{y^i}(u),$$
where each $\psi^i$ is a usual spinor on $M$. We also have the following local expression of $\tilde\nabla$
$$\tilde\nabla\psi=\nabla\psi^i\otimes\partial_{y^i}(u)+\Gamma_{jk}^i(u)\nabla u^j\psi^k\otimes\partial_{y^i}(u),$$
where $\Gamma^i_{jk}$ are the Christoffel symbols of the Levi-Civita connection of $N$. The Dirac operator along the map $u$ is defined as
\begin{equation}\label{dirac}
\slashed{D}:=e_\alpha\cdot\tilde\nabla_{e_\alpha}\psi=\slashed\partial\psi^i\otimes\partial_{y^i}(u)+\Gamma_{jk}^i(u)\nabla_{e_\alpha}u^j(e_\alpha\cdot\psi^k)\otimes\partial_{y^i}(u),
\end{equation}
which is self-adjoint \cite{jost2017riemannian}. Sometimes, we use $\slashed{D}_u$ to distinguish the Dirac operators defined on different maps. In \cite{chen2006dirac}, the authors  introduced the  functional
\begin{equation}\begin{split}
L(u,\psi)&:=\frac12\int_M(|du|^2+\langle\psi,\slashed{D}\psi\rangle_{\Sigma M\otimes u^*TN})\\
&=\frac12\int_M h_{ij}(u)g^{\alpha\beta}\frac{\partial u^i}{\partial x^\alpha}\frac{\pt u^j}{\pt x^\beta}+h_{ij}(u)\langle\psi^i,\slashed{D}\psi^j\rangle_{\Sigma M}.
\end{split}
\end{equation}
They computed the Euler-Lagrange equations of $L$:
\begin{equation}\label{eldh1}
\tau^m(u)-\frac12R^m_{lij}\langle\psi^i,\nabla u^l\cdot\psi^j\rangle_{\Sigma M}=0,
\end{equation}
\begin{equation}\label{eldh2}
\slashed{D}\psi^i=\slashed\pt\psi^i+\Gamma_{jk}^i(u)\nabla_{e_\alpha}u^j(e_\alpha\cdot\psi^k)=0,
\end{equation}
where $\tau^m(u)$ is the $m$-th component of the tension field \cite{jost2017riemannian} of the map $u$ with respect to the coordinates on $N$, $\nabla u^l\cdot\psi^j$ denotes the Clifford multiplication of the vector field $\nabla u^l$ with the spinor $\psi^j$, and $R^m_{lij}$ stands for the component of the Riemann curvature tensor of the target manifold $N$. Denote 
$$\mathcal{R}(u,\psi):=\frac12R^m_{lij}\langle\psi^i,\nabla u^l\cdot\psi^j\rangle_{\Sigma M}\pt_{y^m}.$$ We can write \eqref{eldh1} and \eqref{eldh2} in the following global form:
\begin{numcases}{}
\tau(u)=\mathcal{R}(u,\psi), \label{geldh1} \\
\slashed{D}\psi=0,  \label{geldh2}
\end{numcases}
and call the solutions $(u,\psi)$ Dirac-harmonic maps from $M$ to $N$.

By \cite{nash1956imbedding}, we can isometrically embed $N$ into $\mathbb{R}^q$. Then \eqref{geldh1}-\eqref{geldh2} is equivalent to the system
\begin{numcases}{}
		\Delta_g{u}=II(du,du)+Re(P(\mathcal{S}(du(e_\beta),e_{\beta}\cdot\psi);\psi)), \label{Emap }\\
		\slashed{\partial}\psi=\mathcal{S}(du(e_\beta),e_{\beta}\cdot\psi),	\label{Edirac}
	\end{numcases}
where $II$ is the second fundamental form of $N$ in $\mathbb{R}^q$, and 
\begin{equation}
\mathcal{S}(du(e_\beta),e_{\beta}\cdot\psi):=(\nabla{u^A}\cdot\psi^B)\otimes II(\partial_{z^A},\partial_{z^B}),
\end{equation}
\begin{equation}
Re(P(\mathcal{S}(du(e_\beta),e_{\beta}\cdot\psi);\psi)):=P(S(\partial_{z^C},\partial_{z^B});\partial_{z^A})Re(\langle\psi^A,du^C\cdot\psi^B\rangle).
\end{equation}
Here $A,B,C=1,\cdots,q$, $P(\xi;\cdot)$ denotes the shape operator, defined by $\langle P(\xi;X),Y\rangle=\langle A(X,Y),\xi\rangle$ for $X,Y\in\Gamma(TN)$ and $Re(z)$ denotes the real part of $z\in\mathbb{C}$.

The existence of nontrivial Dirac-harmonic maps is a natural and interesting problem. The following is known. When the surface $M$ has boundary, the non-triviality directly follows from that of the boundary values. In \cite{jost2018geometric}, the authors used the heat flow for $\alpha$-Dirac-harmonic maps to get the existence of $\alpha$-Dirac-harmonic maps which are the critical points of the functional 
 \begin{equation}
 L^\alpha(u,\psi)=\frac12\int_M(1+|du|^2)^\alpha+\frac12\int_M\langle\psi,\slashed{D}^u\psi\rangle_{\Sigma M\otimes u^*TN}
 \end{equation}
with a fixed boundary value. To get the existence of Dirac-harmonic maps, Jost-Liu-Zhu considered the limit behavior of a sequence of $\alpha$-Dirac-harmonic maps. In the context of harmonic maps, this was the method of Sacks-Uhlenbeck in \cite{sacks1981existence}. It is well-known that bubbles generally arise as $\alpha\searrow1$. We also know that bubbles come from the rescaling. Therefore, it is necessary to know how the $\alpha$-Dirac-harmonic equations change under rescaling. Precisely, in isothermal coordinates around a point $x_0\in M$, suppose the metric is given by 
\begin{equation}
h=e^{\varphi_0(x)}((dx^1)^2+(dx^2)^2)
\end{equation}
with $\varphi_0(x_0)=0$; we define
\begin{equation}
(\tilde{u}_\alpha(x),\tilde{v}_\alpha(x)):=(u_\alpha(x_0+\lambda_\alpha x),\sqrt{\lambda_\alpha}\psi_\alpha(x_0+\lambda_\alpha x))
\end{equation}
for a small number $\lambda_\alpha>0$. Then $(\tilde{u}_\alpha,\tilde{v}_\alpha)$ satisfies
\begin{numcases}{}
\begin{split}
		\Delta_{h_\alpha}{\tilde{u}_\alpha}&=-(\alpha-1)\frac{\nabla_{h_\alpha}|\nabla_{h_\alpha}\tilde{u}_\alpha|^2\nabla_{h_\alpha}\tilde{u}_\alpha}{\sigma_\alpha+|\nabla_{h_\alpha}\tilde{u}_\alpha|^2}+II(d\tilde{u}_\alpha,d\tilde{u}_\alpha)\\
		&+\frac{Re(P(\mathcal{S}(d\tilde{u}_\alpha(e_\beta),e_{\beta}\cdot\tilde{v}_\alpha);\tilde{v}_\alpha))}{\alpha(1+\sigma_\alpha^{-1}|\nabla_{h_\alpha}\tilde{u}_\alpha|^2)^{\alpha-1}}, \label{rescaled map}
	\end{split}\\
		\slashed{\partial}_{h_\alpha}\tilde{v}_\alpha=\mathcal{S}(d\tilde{u}_\alpha(e_\beta),e_{\beta}\cdot\tilde{v}_\alpha),	\label{rescaled spinor}
	\end{numcases}
	where $h_\alpha=e^{\varphi_0(x_0+\lambda x)}((dx^1)^2+(dx^2)^2)$ and $\sigma_\alpha=\lambda_\alpha^2>0$. In order to get a useful bubbling equation, it is convenient to add another factor $\lambda_{\alpha}^{\alpha-1}$ in the rescaling, i.e.
	\begin{equation}
	({u}_\alpha(x),{v}_\alpha(x)):=(u_\alpha(x_0+\lambda_\alpha x),\lambda_{\alpha}^{\alpha-1}\sqrt{\lambda_\alpha}\psi_\alpha(x_0+\lambda_\alpha x)).
\end{equation}
Then one can check that $({u}_\alpha(x),{v}_\alpha(x))$ satisfies the  system
\begin{numcases}{}
\begin{split}
		\Delta_{h_\alpha}{{u}_\alpha}&=-(\alpha-1)\frac{\nabla_{h_\alpha}|\nabla_{h_\alpha}{u}_\alpha|^2\nabla_{h_\alpha}{u}_\alpha}{\sigma_\alpha+|\nabla_{h_\alpha}{u}_\alpha|^2}+II(d{u}_\alpha,d{u}_\alpha)\\
		&+\frac{Re(P(\mathcal{S}(d{u}_\alpha(e_\beta),e_{\beta}\cdot v_\alpha);{v}_\alpha))}{\alpha(\sigma_\alpha+|\nabla_{h_\alpha}{u}_\alpha|^2)^{\alpha-1}}, \label{bubble map}
	\end{split}\\
		\slashed{\partial}_{h_\alpha}{v}_\alpha=\mathcal{S}(d{u}_\alpha(e_\beta),e_{\beta}\cdot{v}_\alpha).	\label{bubble spinor}
	\end{numcases}
Once having these equations, it is natural to consider the small energy regularity lemma, an important ingredient in the Sacks-Uhlenbeck scheme. In the case of Dirac-harmonic maps, we need a small energy regularity lemma for  both  systems above.
\begin{lem}\label{small energy regularity}\cite{jost2018geometric}
Let $D_1=D_1(0)\subset\mathbb{R}^2$ be the unit disk with a family of metrics as follows:
\begin{equation}
g_\alpha=e^{\varphi_\alpha(x)}((dx^1)^2+(dx^2)^2),
\end{equation}
where $\varphi_\alpha\in C^\infty(D_1)$, $\varphi_\alpha(0)=0$ and $\varphi_\alpha\to\varphi_0$, $g_\alpha\to g_0$ in $C^\infty(D_1)$ as $\alpha\searrow1$. Suppose that $(u_\alpha,\psi_\alpha):(D_1,g_\alpha)\to N$ satisfies the system \eqref{rescaled map}-\eqref{rescaled spinor} or \eqref{bubble map}-\eqref{bubble spinor} with 
\begin{equation}
0<\beta_0<\liminf_{\alpha\searrow1}\sigma_\alpha^{\alpha-1}\leq1
\end{equation}
for some $\beta_0>0$. For any $1<p<\infty$, there exist two positive constants $\epsilon_0$ and $\alpha_0>1$ depending only on $g_0$ ,$N$, such that if $E_\alpha(u_\alpha)+E(\psi_\alpha)\leq\Lambda$ and
\begin{equation}
E(u_\alpha)\leq\epsilon_0, \ 1<\alpha\leq\alpha_0,
\end{equation}
 then there hold
 \begin{equation}
 \|\nabla u_\alpha\|_{W^{1,p}(D_{1/2})}\leq C(p,g_0,\Lambda,N)\|\nabla u_\alpha\|_{L^2(D_1)},
 \end{equation}
 \begin{equation}
 \|\psi_\alpha\|_{W^{1,p}(D_{1/2})}\leq C(p,g_0,\Lambda,N)\|\psi_\alpha\|_{L^4(D_1)}.
 \end{equation}
\end{lem}

\section{$\alpha$-Dirac-harmonic maps from non-degenerating spin surfaces}
In this section, we will show the generalized energy identity for a sequence of $\alpha$-Dirac-harmonic maps from non-degenerating spin surfaces. This will be used later at the blow-up points away from the punctures of the degenerating surfaces.

For a sequence of $\alpha$-Dirac-harmonic maps from a fixed closed surface, the following generalized energy identity was proved in \cite{jost2018geometric}.
\begin{thm}\label{fixed id}\cite{jost2018geometric}
Let $(u_\alpha,\psi_\alpha)$ be a sequence of smooth $\alpha$-Dirac-harmonic maps from a fixed closed spin surface $(M,h,\mathfrak{S})$ to a compact manifold $N$. If $(u_\alpha,\psi_\alpha)$ satisfies the  uniformly bounded energy condition
\begin{equation}
E_\alpha(u_\alpha)+E(\psi_\alpha)\leq\Lambda,
\end{equation}
then there exist a finite set $S=\{x_1,\cdots,x_I\}$, finitely many Dirac-harmonic spheres $(\sigma^{i,l},\xi^{i,l}): S^2\to N$, $i=1,\cdots,I$, $l=1,\cdots,L_i$ and a Dirac-harmonic map $(u,\psi): (M,h,\mathfrak{S})\to N$ such that, after  selection of a subsequence,
\begin{equation}
(u_\alpha,\psi_\alpha)\to(u,\psi) \ \text{in} \  C^\infty_{loc}(M\setminus S)\times C^\infty_{loc}(M\setminus S)
\end{equation}
and
\begin{equation}
\lim_{\alpha\searrow1}E_\alpha(u_\alpha)=E(u)+|M|+\sum_{i=1}^I\sum_{l=1}^{L_i}\mu_{il}^2E(\sigma^{i,l}),
\end{equation}
\begin{equation}
\lim_{\alpha\searrow1}E(\psi_\alpha)=E(\psi)+\sum_{i=1}^I\sum_{l=1}^{L_i}\mu_{il}^2E(\xi^{i,l}),
\end{equation}
where $\mu_{il}\in[1,\frac{\Lambda}{\epsilon_1}]$ and $\epsilon_1>0$ is the constant such that, if $(\sigma,\xi)$ is a smooth Dirac-harmonic sphere satisfying $\int_{S^2}|d\sigma|^2<\epsilon_1$, then both $\sigma$ and $\xi$ are trivial.

In addition, if $\|\psi_\alpha\|_{L^q}$ is also uniformly bounded for some $q>4$, then bubbles are just harmonic spheres. 
\end{thm}

Theorem \ref{fixed id} directly follows from Lemma \ref{small energy regularity} and the following local model case of a single interior blow-up point.
\begin{lem}\label{local id}
Let $D_1=D_1(0)\subset\mathbb{R}^2$ be the unit disk with a family of metrics as follows:
\begin{equation}
g_\alpha=e^{\varphi_\alpha(x)}((dx^1)^2+(dx^2)^2),
\end{equation}
where $\varphi_\alpha\in C^\infty(D_1)$, $\varphi_\alpha(0)=0$ and $\varphi_\alpha\to\varphi_0$, $g_\alpha\to g_0$ in $C^\infty(D_1)$ as $\alpha\searrow1$. Suppose that  a sequence of solutions $(u_\alpha,\psi_\alpha)\in C^\infty(D_1,N)$ to the system \eqref{rescaled map}-\eqref{rescaled spinor} satisfies:
\begin{equation}
\sup\limits_{\alpha}(E_{\alpha,\sigma_\alpha}(u_\alpha)+E(\psi_\alpha))\leq\Lambda,
\end{equation}
\begin{equation}
0<\beta_0<\liminf_{\alpha\searrow1}\sigma_\alpha^{\alpha-1}\leq1
\end{equation}
and 
\begin{equation}
(u_\alpha,\psi_\alpha)\to(u,\psi) \ \text{in} \ C^\infty_{loc}(D_1\setminus\{0\}) \ \text{as} \ \alpha\searrow1.
\end{equation}
Then there exist a subsequence of $(u_\alpha,\psi_\alpha)$ (still denoted by $(u_\alpha,\psi_\alpha)$) and a nonnegative integer $L_1$ such that for any $i=1,\cdots,L_1$, there exist sequences of positive numbers $\lambda^i_\alpha$ and nontrivial Dirac-harmonic spheres $(\sigma^i,\xi^i)$ such that 
\begin{equation}
\lim_{\delta\to0}\lim_{\alpha\searrow1}E_{\alpha,\sigma_\alpha}(u_\alpha,D_\delta)=\sum_{i=1}^{L_1}\mu_{i}^2E(\sigma^i),
\end{equation}
\begin{equation}
\lim_{\delta\to0}\lim_{\alpha\searrow1}E(\psi_\alpha,D_\delta)=\sum_{i=1}^{L_1}\mu_{i}^2E(\xi^i),
\end{equation}
where $\mu_i=\lim\limits_{\alpha\searrow1}(\lambda^i_\alpha)^{2-2\alpha}$.
\end{lem}

It is then natural to ask what happens when the domain of $\alpha$-Dirac-harmonic maps $(u_\alpha,\psi_\alpha)$ varies. In this section, we consider the  simplest case where $(M_n,h_n,c_n)$ converges to a closed hyperbolic Riemann surface $(M,h,c)$. By the pull-forward discussed in the previous section, we view $(u_n,\psi_n):=(u_{\alpha_n},\psi_{\alpha_n})$ as a sequence of $\alpha_n$-Dirac-harmonic maps defined on $(M,\bar{h}_n,\bar{c}_n,\mathfrak{S})$ with respect to $(\bar{c}_n,\nabla_n)$, where $\nabla_n$ is the connection on the spinor bundle $\Sigma M$ coming from $\bar{h}_n$ and $\alpha_n\searrow1$. Then we have the following energy identity:
\begin{thm}\label{nondegenerate}
With the  notations  above, we assume that $(u_n,\psi_n)$ satisfies the uniformly bounded energy assumption 
\begin{equation}
E_{\alpha_n}(u_n,\bar{h}_n)+E(\psi_n,\bar{h}_n)\leq\Lambda.
\end{equation}
Then 
there exist a finite set $S=\{x_1,\cdots,x_I\}$, finitely many Dirac-harmonic spheres $(\sigma^{i,l},\xi^{i,l}): S^2\to N$, $i=1,\cdots,I$, $l=1,\cdots,L_i$ and a Dirac-harmonic map $(u,\psi): (M,h,\mathfrak{S})\to N$ such that, after  selection of a subsequence,
\begin{equation}
(u_n,\psi_n)\to(u,\psi) \ \text{in} \  C^\infty_{loc}(M\setminus S)\times C^\infty_{loc}(M\setminus S)
\end{equation}
and
\begin{equation}
\lim_{n\to\infty}E_{\alpha_n}(u_n,\bar{h}_n)=E(u,h)+|M|_h+\sum_{i=1}^I\sum_{l=1}^{L_i}\mu_{il}^2E(\sigma^{i,l}),
\end{equation}
\begin{equation}
\lim_{n\to\infty}E(\psi_n,\bar{h}_n)=E(\psi,h)+\sum_{i=1}^I\sum_{l=1}^{L_i}\mu_{il}^2E(\xi^{i,l}),
\end{equation}
where $\mu_{il}\in[1,\frac{\Lambda}{\epsilon_1}]$ and $\epsilon_1>0$ is the constant such that, if $(\sigma,\xi)$ is a smooth Dirac-harmonic sphere satisfying $\int_{S^2}|d\sigma|^2<\epsilon_1$, then both $\sigma$ and $\xi$ are trivial.
\end{thm}

\begin{proof}
Since $(\bar{c}_n,\nabla_n)\to(h,c)$ in $C^\infty(M)$ as $n\to\infty$, all the geometric data converges in $C^\infty(M)$. In particular, $\nabla_n-\nabla\to0$ in $C^\infty(M)$, where $\nabla$ is the connection on $\Sigma M$ coming from $h$. Therefore, by the uniformly bounded energy assumption, we can assume $(u_n,\psi_n)$ weakly converges to some $(u,\psi)$ in $W^{1,2}(M,N)\times L^4(\Sigma M\otimes\mathbb{R}^K)$ with respect to $(c,\nabla)$, where we have isometrically embedded $N$ into $\mathbb{R}^K$. Note that all the constants in the small energy regularity lemma (Lemma \ref{small energy regularity}) and the local version of generalized energy identity (Lemma \ref{local id}) are uniform with respect to $\bar{h}_n$ and $\bar{c}_n$. Hence, we have the  generalized energy identity.
 \end{proof}

Moreover, it follows from the following lemma that $\|\psi_\alpha\|_{L^q}$ is also uniformly bounded for some $q>4$. Therefore, the bubbles are just harmonic spheres.
 \begin{lem}\label{spinor Lq norm}
 Let $M$ be a compact spin Riemann surface with boundary $\partial M$, N be a compact Riemann manifold. Let $u\in W^{1,2\alpha}(M,N)$ for some $\alpha>1$ and $\psi\in W^{1,p}(M, \Sigma M\otimes u^*TN)$ for $1<p<2$, then there exists a positive constant $C=C(p,M,N,\|\nabla u\|_{L^{2\alpha}})$ such that
 \begin{equation}
 \|\psi\|_{W^{1,p}(M)}\leq C(\|\slashed{D}\psi\|_{L^{p}(M)}+\|\psi\|_{L^p(M)}).
 \end{equation}
 \end{lem}
 
 \begin{proof}
Applying the following lemma to $\eta\psi$ for some cut-off $\eta$, we complete the proof.
\end{proof}
 \begin{lem}\label{spinor norm with boundary}\cite{jost2018geometric}
 Let $M$ be a compact spin Riemann surface with boundary $\partial M$, N be a compact Riemann manifold. Let $u\in W^{1,2\alpha}(M,N)$ for some $\alpha>1$ and $\psi\in W^{1,p}(M, \Sigma M\otimes u^*TN)$ for $1<p<2$, then there exists a positive constant $C=C(p,M,N,\|\nabla u\|_{L^{2\alpha}})$ such that
 \begin{equation}
 \|\psi\|_{W^{1,p}(M)}\leq C(\|\slashed{D}\psi\|_{L^{p}(M)}+\|\mathcal{B}\psi\|_{W^{1-1/p,p}(\partial M)}),
 \end{equation}
 where $\mathcal{B}$ is the Chiral boundary operator for spinors along a map.
 \end{lem}

\section{Dirac-harmonic maps from degenerating spin surfaces}
We divide this section into two parts. In the first part, we discuss the compactness of a sequence of $\alpha$-(Dirac-)harmonic maps from closed Riemann (spin) surfaces degenerating to a hyperbolic Riemann (spin) surface. In the second part, based on the compactness result in the first part, we prove an existence result about  (Dirac-)harmonic maps from degenerating surfaces.

\subsection{Compactness and Energy identity }

The following theorem is the main result of this subsection.
 \begin{thm}\label{id dh}
Let $(M_n,h_n,c_n,\mathfrak{S}_n)$ be a sequence of closed hyperbolic surfaces of genus $g>1$ degenerating to a hyperbolic Riemann surface $(M,h,c,\mathfrak{S})$ by collapsing finitely many pairwise disjoint simple closed geodesics $\{\gamma_n^j, j\in J\}$. Denote by $l_n^j$ the length of $\gamma_n^j$ and $\l_n:=\max\limits_{j\in J}\{l_n^j\}$. We choose a sequence of constants, $\{\alpha_n\in(1,2)\}$, such that 
\begin{equation}\label{choice of alpha}
\lim_{n\to\infty}\bigg(\frac{2l_n}{\sqrt{\sinh{\frac{l_n}{2}}}}\bigg)^{\alpha_n-1}=0.
\end{equation}
For each $n$, suppose that $(u_n,\psi_n)$ is an $\alpha_n$-Dirac-harmonic map from $(M_n,h_n,c_n,\mathfrak{S})$ into a fixed compact manifold $N$ with nonpositive curvature which satisfies
\begin{equation}
E_{\alpha_n}(u_n)+E(\psi_n)\leq\Lambda,
\end{equation}
 for some positive constant $\Lambda$. Moreover, we assume that the first positive eigenvalue $ \lambda_1(h_n)$ of the usual Dirac operator $\slashed{\partial}_{h_n}$ stays away from zero, i.e.
 \begin{equation}\label{1st eigenvalue}
 \lambda_1(h_n)\geq c_0
 \end{equation}
 for some positive constant $c_0>0$. Then there exists a Dirac-harmonic map $(u,\psi):(M,h,c,\mathfrak{S})\to N$ such that, after  selection of a subsequence,
 \begin{equation}\label{loc convergence}
 (u_n,\psi_n)\to(u,\psi) \ \text{in} \ C^{\infty}_{loc}(M)\times C^{\infty}_{loc}(M)
 \end{equation}
 and 
 \begin{equation}\label{cid map}
 \lim_{n\to\infty}E(u_n,h_n,M_n)=E(u,h,M),
 \end{equation}
  \begin{equation}\label{cid spinor}
 \lim_{n\to\infty}E(\psi_n,h_n,M_n)=E(\psi,h,M).
 \end{equation}
 \end{thm}
 
 \begin{proof}
We first consider the  case  $|J|=1$. For $0<\delta<\arcsinh 1$, let 
\begin{equation}
M^{\delta}=\{x\in M, {\rm inj}(z;h)\geq\delta\}
\end{equation} be the $\delta$-thick part of the hyperbolic surface $(M,h)$. As explained in Section 2, there are diffeomorphisms $\tau_n: M_n\setminus\gamma_n\to M$ such that $((\tau_n)_*h_n,(\tau_n)_*c_n)\to(h,c)$ in $C^\infty_{loc}(M)$. We set 
\begin{equation}
\bar{u}_n:=(\tau_n)_*u_n,  \  \bar{v}_n:=(\tau_n)_*\psi_n, \  \bar{h}_n:=(\tau_n)_*h_n, \  \bar{c}_n:=(\tau_n)_*c_n
\end{equation}
and consider the following sequence of $\alpha_n$-Dirac-harmonic maps
\begin{equation}
(\bar{u}_n,\bar{v}_n): (\Sigma,\bar{h}_n,\bar{c}_n,\mathfrak{S})\to N.
\end{equation}
Then, for each fixed $\delta>0$, we have
\begin{equation}
(\bar{h}_n,\bar{c}_n)\to(h,c) \ \text{in} \ C^{\infty}(M^\delta).
\end{equation}
We choose a  sequence $\delta_n\searrow0$ such that $M^{\delta_n}$ exhaust $M$. By Theorem \ref{nondegenerate} and a standard diagonal argument, there exists a Dirac-harmonic map $(u,\psi):(M,h,c)\to N$ such that the following hold
 \begin{equation}
  \lim_{n\to\infty}E_{\alpha_n}(\bar{u}_n,\bar{h}_n,M^{\delta_n})=E(u,h,M)+|M|_h+\sum_{i=1}^I\sum_{l=1}^{L_i}\mu_{il}^2E(\sigma^{i,l}),
\end{equation}
\begin{equation}
\lim_{n\to\infty}E(\bar{v}_n,\bar{h}_n,M^{\delta_n})=E(\psi,h,M)+\sum_{i=1}^I\sum_{l=1}^{L_i}\mu_{il}^2E(\xi^{i,l}).
\end{equation}
Moreover, by Lemma \ref{spinor Lq norm} and our assumption on the target manifold $N$, we get \eqref{loc convergence} and 
 \begin{equation}\label{away from degenerate map}
 \lim_{n\to\infty}E(u_n,h_n,\tau_n^{-1}(M^{\delta_n}))=\lim_{n\to\infty}E(\bar{u}_n,\bar{h}_n,M^{\delta_n})=E(u,h,M),
\end{equation}
\begin{equation}\label{away from degenerate spinor}
 \lim_{n\to\infty}E(\psi_n,h_n,\tau_n^{-1}(M^{\delta_n}))=\lim_{n\to\infty}E(\bar{v}_n,\bar{h}_n,M^{\delta_n})=E(\psi,h,M).
\end{equation}

To recover the energy concentration at the punctures $(\mathcal{P}^1,\mathcal{P}^2)$, we need to study $(\bar{u}_n,\bar{v}_n)$ on $M\setminus M^{\delta_n}$, or equivalently $(u_n,\psi_n)$ on $M_n\setminus\tau_n^{-1}(M^{\delta_n})$. For each $n,\delta$, $M_n\setminus\tau_n^{-1}(M^\delta)$ is not the $\delta$-thin part of $(M_n,h_n)$. However, for fixed $\delta>0$ and sufficiently large $n$, $M_n\setminus\tau_n^{-1}(M^\delta)$ is almost the $\delta$-thin part of $(M_n,h_n)$.

To see this, fix $\delta>0$ small and let $x\in M$ be a point with ${\rm inj}(x;h)=\delta$. Since $\bar{h}_n\to h$ in $C^\infty_{loc}$ on $M$, for any $\delta_1,\delta_2>0$ such that $\delta_1<\delta<\delta_2$, we have
\begin{equation}\label{inj estimate}
\delta_1<{\rm inj}(x;\bar{h}_n)<\delta_2 \ \text{for large} \ n.
\end{equation}
Recall that for $0<\delta<\arcsinh 1$, the $\delta$-thin part of a hyperbolic surface is either an annulus or a cusp. For $n\geq1$ and $\delta\in[\frac{l_n}{2},\arcsinh 1]$, let us see what the $\delta$-thin part of $(M_n,h_n)$ looks like. Recall that $P_n$ is the cylinder collar about $\gamma_n$. Now, we define the following $\delta$-subcollar of $P_n$
 \begin{equation}
 P_n^{\delta}:=[T_n^{1,\delta},T_n^{2,\delta}]\times S^1\subset P_n,
 \end{equation}
 where 
 \begin{equation}
 T_n^{1,\delta}=\frac{2\pi}{l_n}\arcsin\bigg(\frac{\sinh(\frac{l_n}{2})}{\sinh\delta}\bigg), \  T_n^{2,\delta}=\frac{2\pi^2}{l_n}-\frac{2\pi}{l_n}\arcsin\bigg(\frac{\sinh(\frac{l_n}{2})}{\sinh\delta}\bigg).
 \end{equation}
 By \eqref{inj}, $P_n^{\delta}$ is exactly the $\delta$-thin part of $(M_n,h_n)$, namely
 \begin{equation}\label{delta thin}
 P_n^\delta=\{x\in M_n: {\rm inj}(x;h_n)\leq\delta\}.
 \end{equation} 
 Thus, it follows from \eqref{inj estimate} and \eqref{delta thin} that 
 \begin{equation}\label{almost}
 P_n^{\delta_1}\subset M_n\setminus\tau_n^{-1}(M^\delta)\subset P_n^{\delta_2} \ \text{for all} \ n \ \text{large enough}. 
 \end{equation}
 If we choose $\delta_1,\delta_2$ in \eqref{almost} sufficiently close to $\delta$, then for $n$ large enough, $M_n\setminus\tau_n^{-1}(M^\delta)$ is almost the $\delta$-thin part $P_n^\delta$ of $(M_n,h_n)$.
 
 Now, for $\delta>0$ small and $n$ large enough, we define
 \begin{equation}
 \Omega_n^\delta:=\{(M_n\setminus\tau_n^{-1}(M^\delta))\setminus P_n^\delta\}\cup\{P_n^\delta\setminus(M_n\setminus\tau_n^{-1}(M^\delta))\}.
 \end{equation}
 Then the image of $\Omega_n^\delta$ under $\tau_n$ is uniformly away from the punctures of $M$ and actually converges to $\partial M^\delta$. Therefore, we have
 \begin{equation}
 \lim_{n\to\infty}E(u_n,\Omega_n^\delta)=0,
 \end{equation}
  \begin{equation}
 \lim_{n\to\infty}E(\psi_n,\Omega_n^\delta)=0.
 \end{equation}
Thus, after passing to a subsequence, we conclude that 
  \begin{equation}\label{reduce map}
 \lim_{n\to\infty}E(u_n,(M_n\setminus\tau_n^{-1}(M^{\delta_n}))= \lim_{n\to\infty}E(u_n,P_n^{\delta_n}),
 \end{equation}
  \begin{equation}\label{reduce spinor}
 \lim_{n\to\infty}E(\psi_n,(M_n\setminus\tau_n^{-1}(M^{\delta_n}))= \lim_{n\to\infty}E(\psi_n,P_n^{\delta_n}).
 \end{equation}
 For the right-hand side of \eqref{reduce map}-\eqref{reduce spinor}, by the H$\rm\ddot{o}$lder inequality, we get
 \begin{equation}\label{map energy cylinder}
 \begin{split}
 \lim_{n\to\infty}E(u_n,P_n^{\delta_n})&=\lim_{n\to\infty}\int_{P_n^{\delta_n}}|du_n|^2\\
 &\leq\lim_{n\to\infty}(\int_{P_n^{\delta_n}}|du_n|^2)^{\frac{1}{\alpha_n}}({\rm Area}(P_n^{\delta_n}))^{1-\frac{1}{\alpha_n}}\\
 &\leq\Lambda\lim_{n\to\infty}({\rm Area}(P_n^{\delta_n}))^{1-\frac{1}{\alpha_n}}
 \end{split}
 \end{equation}
 and
 \begin{equation}\label{spinor cylinder}
 \lim_{n\to\infty}E(\psi_n,P_n^{\delta_n})=\lim_{n\to\infty}\int_{P_n^{\delta_n}}|\psi_n|^4\leq\lim_{n\to\infty}\|\psi_n\|_{L^q(M_n)}^4({\rm Area}(P_n^{\delta_n}))^{1-\frac{4}{q}}
\end{equation}
 for some $q>4$. By Lemma \ref{spinor Lq norm}, we have
 \begin{equation}
 \|\psi_n\|_{L^q(M_n)}\leq C(q,M_n,N,\Lambda).
 \end{equation}
 Note that when $(M_n,h_n)$ varies, the constant above actually depends on the lower bound of $\lambda_1(h_n)$. Therefore, by our assumption \eqref{1st eigenvalue} on the first eigenvalue of $\slashed{\partial}_{h_n}$, we have a uniform bound for $\|\psi_n\|_{L^q(M_n)}$, and \eqref{spinor cylinder} becomes 
  \begin{equation}\label{spinor energy cylinder}
 \lim_{n\to\infty}E(\psi_n,P_n^{\delta_n})\leq C_0\lim_{n\to\infty}({\rm Area}(P_n^{\delta_n}))^{1-\frac{4}{q}}
\end{equation}
 for some $C_0>0$.
 
 It remains to consider $\lim\limits_{n\to\infty}({\rm Area}(P_n^{\delta_n}))^{1-\frac{1}{\alpha_n}}
$. By the definition of $P_n^{\delta_n}$, the area of $P_n^{\delta_n}$ can be computed as
\begin{equation}
\begin{split}
{\rm Area}(P_n^{\delta_n})&=\int_{T_n^{1,\delta_n}}^{T_n^{2,\delta_n}}\int_0^{2\pi}\bigg(\frac{l_n}{2\pi\sin(\frac{l_nt}{2\pi})}\bigg)^2dtd\theta\\
&=2\pi\int_{T_n^{1,\delta_n}}^{T_n^{2,\delta_n}}\frac{l_n^2}{4\pi^2}\frac{1}{\sin^2(\frac{l_nt}{2\pi})}dt\\
&=2\pi\int_{\frac{l_n}{2\pi}T_n^{1,\delta_n}}^{\frac{l_n}{2\pi}T_n^{2,\delta_n}}\frac{l_n^2}{4\pi^2}\frac{1}{\sin^2s}\frac{2\pi}{l_n}ds\\
&=l_n\int_{\frac{l_n}{2\pi}T_n^{1,\delta_n}}^{\frac{l_n}{2\pi}T_n^{2,\delta_n}}\frac{1}{\sin^2s}ds\\
&=-l_n(\cot(\pi-\arcsin\varphi)-\cot(\arcsin\varphi))\\
&=2l_n\sqrt{\frac{\sinh\delta_n}{\sinh{\frac{l_n}{2}}}-1}\\
&\leq\frac{2l_n}{\sqrt{\sinh{\frac{l_n}{2}}}},
\end{split}\end{equation}
where $s=\frac{l_nt}{2\pi}$ and $\varphi=\frac{\sinh{\frac{l_n}{2}}}{\sinh\delta_n}$. Thus,
\begin{equation}\label{area spinor}
\begin{split}
\lim_{n\to\infty}({\rm Area}(P_n^{\delta_n}))^{1-\frac{4}{q}}&=(\lim_{n\to\infty}{\rm Area}(P_n^{\delta_n}))^{1-\frac{4}{q}}\\
&\leq(\lim_{n\to\infty}\frac{2l_n}{\sqrt{\sinh{\frac{l_n}{2}}}})^{1-\frac{4}{q}}\\
&=(\lim_{n\to\infty}8\sqrt{\sinh{\frac{l_n}{2}}})^{1-\frac{4}{q}}=0
\end{split}
\end{equation}
and 
\begin{equation}\label{area map}
\begin{split}
\lim_{n\to\infty}({\rm Area}(P_n^{\delta_n}))^{1-\frac{1}{\alpha_n}}&=\lim_{n\to\infty}({\rm Area}(P_n^{\delta_n}))^{\alpha_n-1}\\
&\leq\lim_{n\to\infty}(\frac{2l_n}{\sqrt{\sinh{\frac{l_n}{2}}}})^{\alpha_n-1}\\
&=0,
\end{split}
\end{equation}
where we have used the assumption \eqref{choice of alpha} in the last equality. Plugging \eqref{area spinor} and \eqref{area map} into \eqref{spinor energy cylinder} and \eqref{map energy cylinder}, we get
 \begin{equation}\label{energy cylinder map} 
 \lim_{n\to\infty}E(u_n,P_n^{\delta_n})=0
  \end{equation}
  and
   \begin{equation}\label{energy cylinder spinor}
 \lim_{n\to\infty}E(\psi_n,P_n^{\delta_n})=0.
\end{equation}
Last, by combining \eqref{away from degenerate map}-\eqref{away from degenerate spinor}, \eqref{reduce map}-\eqref{reduce spinor} and \eqref{energy cylinder map}-\eqref{energy cylinder spinor}, we have \eqref{cid map}-\eqref{cid spinor} in the case of $|J|=1$. By the thick-thin decomposition of hyperbolic surfaces in Section 2, both the short simple closed geodesics of lengths less than $2\arcsinh 1$ and the corresponding $(\arcsinh 1)$-thin parts of the collars around them are pairwisely disjoint, Hence we can deal with the corresponding subcollars separately, and the  case just studied applies. This completes the proof.

 \end{proof}

The preceding proof directly yields a similar theorem for harmonic maps.
\begin{thm}\label{id h}
Let $(\Sigma_n,h_n,c_n)$ be a sequence of closed hyperbolic surfaces of genus $g>1$ degenerating to a hyperbolic Riemann surface $(\Sigma,h,c)$ by collapsing finitely many pairwise disjoint simple closed geodesics $\{\gamma_n^j, j\in J\}$. Denote by $l_n^j$ the length of $\gamma_n^j$ and $\l_n:=\max\limits_{j\in J}\{l_n^j\}$. We choose a sequence of constants, $\{\alpha_n\in(1,2)\}$, such that 
\begin{equation}
\lim_{n\to\infty}\bigg(\frac{2l_n}{\sqrt{\sinh{\frac{l_n}{2}}}}\bigg)^{\alpha_n-1}=0.
\end{equation}
For each $n$, suppose that $u_n$ is an $\alpha_n$-harmonic map from $(\Sigma_n,h_n,c_n)$ into a fixed compact manifold $N$ with nonpositive curvature which satisfies
\begin{equation}
E_{\alpha_n}(u_n)\leq\Lambda,
\end{equation}
 for some positive constant $\Lambda$. Then there exists a harmonic map $u:(\Sigma,h,c)\to N$ such that, after  selection of a subsequence,
 \begin{equation}
 u_n\to u\ \text{in} \ C^{\infty}_{loc}(\Sigma)
 \end{equation}
 and 
 \begin{equation}
 \lim_{n\to\infty}E(u_n,h_n,\Sigma_n)=E(u,h,\Sigma).
 \end{equation}
\end{thm} 

We conclude this subsection with  a remark on the two theorems above. First, different from \cite{zhu2009dirac}, we do not need to restrict the type of degeneration in Theorem \ref{id dh}. Second, it follows from our cleaner energy identity that the limit map $u$ (or $u$) is nontrivial under the necessary condition $\lim\limits_{n\to\infty}E(u_n)\neq0$ (or $\lim\limits_{n\to\infty}E(u_n)\neq0$). Last, for fixed $\alpha>1$, a sequence of $\alpha$-Dirac-harmonic maps always has a convergent subsequence with no bubbles (see \cite{jost2019alpha}), and one can similarly prove the following theorems for $\alpha$-(Dirac-)harmonic maps from degenerating surfaces into an arbitrary compact target manifold $N$.

\begin{thm}\label{id h alpha}
Let $(\Sigma_n,h_n,c_n)$ be a sequence of closed hyperbolic surfaces of genus $g>1$ degenerating to a hyperbolic Riemann surface $(\Sigma,h,c)$ by collapsing finitely many pairwise disjoint simple closed geodesics $\{\gamma_n^j, j\in J\}$. Denote by $l_n^j$ the length of $\gamma_n^j$ and $\l_n:=\max\limits_{j\in J}\{l_n^j\}$. For each $n$, suppose that $u_n$ is a $\alpha(>1)$-harmonic map from $(\Sigma_n,h_n,c_n)$ into a fixed compact manifold $N$ which satisfies
\begin{equation}
E_{\alpha}(u_n)\leq\Lambda,
\end{equation}
 for some positive constant $\Lambda$. Then there exists an $\alpha$-harmonic map $u:(\Sigma,h,c)\to N$ such that, after  selection of a subsequence,
 \begin{equation}
 u_n\to u\ \text{in} \ C^{\infty}_{loc}(\Sigma)
 \end{equation}
 and 
 \begin{equation}
 \lim_{n\to\infty}E(u_n,h_n,\Sigma_n)=E(u,h,\Sigma).
 \end{equation}
\end{thm} 

 \begin{thm}\label{id dh alpha}
Let $(M_n,h_n,c_n,\mathfrak{S}_n)$ be a sequence of closed hyperbolic surfaces of genus $g>1$ degenerating to a hyperbolic Riemann surface $(M,h,c,\mathfrak{S})$ by collapsing finitely many pairwise disjoint simple closed geodesics $\{\gamma_n^j, j\in J\}$. Denote by $l_n^j$ the length of $\gamma_n^j$ and $\l_n:=\max\limits_{j\in J}\{l_n^j\}$. For each $n$, suppose that $(u_n,\psi_n)$ is a $\alpha(>1)$-Dirac-harmonic map from $(M_n,h_n,c_n,\mathfrak{S}_n)$ into a fixed compact manifold $N$ which  satisfies
\begin{equation}
E_{\alpha}(u_n)+E(\psi_n)\leq\Lambda,
\end{equation}
 for some positive constant $\Lambda$. Then there exists an $\alpha$-Dirac-harmonic map $(u,\psi):(M,h,c,\mathfrak{S})\to N$ such that, after  selection of a subsequence,
 \begin{equation}
 (u_n,\psi_n)\to(u,\psi) \ \text{in} \ C^{\infty}_{loc}(M)\times C^{\infty}_{loc}(M)
 \end{equation}
 and 
 \begin{equation}
 \lim_{n\to\infty}E(u_n,h_nM_n)=E(u,h,M),
 \end{equation}
  \begin{equation}
 \lim_{n\to\infty}E(\psi_n,h_n,M_n)=E(\psi,h,M).
 \end{equation}
 \end{thm}

\subsection{Existence}
Theorems \ref{id dh} and \ref{id h} yield an existence result for (Dirac-)harmonic maps from degenerating surfaces if the corresponding assumptions are satisfied. In this subsection, we will realize those assumptions based on the existence results in \cite{ammann2013dirac} and \cite{sacks1981existence}.

For $\alpha$-Dirac-harmonic maps from a closed surface, we have the following existence result:
\begin{thm}\label{uncoupled alpha DH map}
Let $M$ be a closed spin surface and $N$ a compact manifold. Suppose there exists a map $u_0\in C^{2+\mu}(M,N)$ for some $\mu\in(0,1)$ such that 
\begin{equation}
{\rm ind}_{u_0^*TN}(M)=[{\rm dim}_{\mathbb{H}}{\rm ker}\slashed{D}^{u_0}]_{\mathbb{Z}_2}\neq0.
\end{equation} 
Then for any $\alpha\geq1$ and any $\alpha$-harmonic map $u_\alpha$ in the homotopy class $[u_0]$, there exists a nontrivial smooth $\alpha$-Dirac-harmonic map $(u_\alpha,\psi_\alpha)$ such that $\|\psi_\alpha\|_{L^2}=1$. 
\end{thm}
 The proof is similar to the one in \cite{ammann2013dirac} for Dirac-harmonic maps, see also \cite{jost2019short}.
 
 \begin{proof}
 If $\slashed{D}^{u_\alpha}$ has minimal kernel, that is, ${\rm dim}_{\mathbb{H}}{\rm ker}\slashed{D}^{u_0}=1$, then for any $\psi\in{\rm ker}{\slashed{D}^{u_\alpha}}$, $(u_\alpha,\psi)$ is an $\alpha$-Dirac-harmonic map by Proposition 8.2 in \cite{ammann2013dirac} (see also \cite{jost2019short} for a proof by the heat flow). If $\slashed{D}^{u_\alpha}$ has non-minimal kernel, we use the decomposition of the twisted spinor bundle through the $\mathbb{Z}_2$-grading $G\otimes{\rm id}$ (see \cite{ammann2013dirac}). More precisely, for any smooth variation $(u_s)_{s\in(-\epsilon,\epsilon)}$ of $u_0$, we split the bundle $\Sigma M\otimes u_s^*TN$ into $\Sigma M\otimes u_s^*TN=\Sigma^+M\otimes u_s^*TN\oplus\Sigma^-M\otimes u_s^*TN$, which is orthogonal in the complex sense and parallel. Consequently, for any $\psi_0\in{\rm ker}{\slashed{D}^{u_0}}$, we have
\begin{equation} 
(\slashed{D}^{u_0}\psi_0^+,\psi_0^+)_{L^2}=(\slashed{D}^{u_0}\psi_0^-,\psi_0^-)_{L^2}=0
\end{equation}
for $\psi_0=\psi_0^++\psi_0^-$, where $\psi_0^{\pm}=\psi_{\pm}\otimes u_0^*TN$ and $\psi_{\pm}\in\Sigma^{\pm}$. Therefore, $\psi_s^{\pm}:=\psi_{\pm}\otimes u_s^*TN$ are smooth variations of $\psi_0^\pm$, respectively, such that 
\begin{equation}
\frac{d}{dt}\bigg|_{t=0}(\slashed{D}^{u_s}\psi_s^\pm,\psi_s^\pm)_{L^2}=0.
\end{equation} 
By taking $u_0=u_\alpha$ and $\psi_0=\psi_\alpha\in{\rm ker}{\slashed{D}^{u_\alpha}}$, the first variation formula of $L^\alpha$ implies that $(u_\alpha,\psi_\alpha^\pm)$ are $\alpha$-Dirac-harmonic maps (see Corollary 5.2 in \cite{ammann2013dirac}). In particular, we can choose $\psi_\alpha$ such that $\|\psi_\alpha^+\|_{L^2}=1$ or $\|\psi_\alpha^-\|=1$.
 \end{proof}

Note that ${\rm ind}_{u_0^*TN}(M)$ is independent of the choice of the Riemannian metrics on $M$ and $N$. It is also invariant in the homotopy class $[u_0]$. Combining these facts and the results in the previous subsection, we directly get the following existence results.

 \begin{thm}
Let $(M_n,h_n,c_n,\mathfrak{S}_n)$ be a sequence of closed hyperbolic surfaces of genus $g>1$ degenerating to a hyperbolic Riemann surface $(M,h,c,\mathfrak{S})$ by collapsing finitely many pairwise disjoint simple closed geodesics $\{\gamma_n^j, j\in J\}$. Denote by $l_n^j$ the length of $\gamma_n^j$ and $\l_n:=\max\limits_{j\in J}\{l_n^j\}$. For each $n$, suppose that $u_n$ is a map from $M_n$ into a fixed compact manifold $N$ with nonpositive curvature and satisfies
\begin{equation}
E_{\alpha_n}(u_n)\leq\Lambda,
\end{equation}
\begin{equation}
{\rm ind}_{u_n^*TN}(M_n)\neq0,
\end{equation}
where $\Lambda$ is a positive constant, and $\{\alpha_n\in(1,2)\}$ satisfies
\begin{equation}
\lim_{n\to\infty}\bigg(\frac{2l_n}{\sqrt{\sinh{\frac{l_n}{2}}}}\bigg)^{\alpha_n-1}=0.
\end{equation}
Moreover, we assume that the first positive eigenvalue $ \lambda_1(h_n)$ of the usual Dirac operator $\slashed{\partial}_{h_n}$ stays away from zero, i.e.
 \begin{equation}
 \lambda_1(h_n)\geq c_0
 \end{equation}
 for some positive constant $c_0>0$. Then there exists a Dirac-harmonic map $(u,\psi):(M,h,c,\mathfrak{S})\to N$ such that
  \begin{equation}
\psi\neq0.
 \end{equation}
 In addition, let   the minimizing harmonic map $u_n^{\min}$ in $[u_n]$  satisfy
 \begin{equation}
 \lim_{n\to\infty}E(u_n^{\min})>0.
 \end{equation}
 Then $u$ is not a constant.
 \end{thm}
 
\begin{proof}
By our assumptions, we get a sequence $\alpha$-Dirac-harmonic maps satisfying the assumption in Theorem \ref{id dh}. Therefore, there exists a Dirac-harmonic map $(u,\psi)$ from $M$ to $N$. Combining the energy identity \eqref{cid map}-\eqref{cid spinor} and Theorem \ref{uncoupled alpha DH map}, we have the non-triviality of $u$ and $\psi$.
\end{proof}

Similarly, based on the existence of $\alpha$-harmonic maps in \cite{sacks1981existence}, we obtain an analogous theorem for harmonic maps.
 \begin{thm}
Let $(\Sigma_n,h_n,c_n)$ be a sequence of closed hyperbolic surfaces of genus $g>1$ degenerating to a hyperbolic Riemann surface $(\Sigma,h,c)$ by collapsing finitely many pairwise disjoint simple closed geodesics $\{\gamma_n^j, j\in J\}$. Denote by $l_n^j$ the length of $\gamma_n^j$ and $\l_n:=\max\limits_{j\in J}\{l_n^j\}$. For each $n$, suppose that $u_n$ is a map from $\Sigma_n$ into a fixed compact manifold $N$ with nonpositive curvature which satisfies
\begin{equation}
E_{\alpha_n}(u_n)\leq\Lambda,
\end{equation}
where $\Lambda$ is a positive constant, and $\{\alpha_n\in(1,2)\}$ satisfies
\begin{equation}
\lim_{n\to\infty}\bigg(\frac{2l_n}{\sqrt{\sinh{\frac{l_n}{2}}}}\bigg)^{\alpha_n-1}=0.
\end{equation}
 Then there exists a harmonic map $u_0:(\Sigma,h,c)\to N$, which can be extended to a smooth harmonic map on $(\bar{\Sigma},\bar{c})$.
 
 In addition, let $u_n^{\min}$ be the minimizing harmonic map in $[u_n]$ and let it satisfy
 \begin{equation}
 \lim_{n\to\infty}E(u_n^{\min})>0.
 \end{equation}
 Then $u_0$ is not a constant.
 \end{thm}

% ----------------------------------------------------------------

\nocite{*}

% ----------------------------------------------------------------

\bibliographystyle{amsplain}
\bibliography{reference}

\end{document}